\documentclass[12pt,leqno]{amsart}
\usepackage{amsmath,amssymb,txfonts}

      \newtheorem{theorem}{Theorem}[section]
      \newtheorem{definition}[theorem]{Definition}
      
      \newtheorem{corollary}[theorem]{Corollary}
      
      \newtheorem{lemma}[theorem]{Lemma}

      \newcommand{\ct}[1]{\langle {#1}\rangle \lower.3ex\hbox{$_{t}$}}
      \newcommand{\lt}[1]{[ {#1}] \lower.3ex\hbox{$_{t}$}}

\begin{document}

\title[Two Predualities and Three Operators over Analytic Campanato Spaces]{Two Predualities and Three Operators over Analytic Campanato Spaces}
\author{Jianfei Wang and Jie Xiao}
\address{Mathematics and Physics, Information Engineering, Zhejiang Normal
University, Jinhua, Zhejiang, 321004, P. R. China}
\curraddr{Department of Mathematics and Statistics, Memorial
University, St. John's, NL A1C 5S7, Canada}
\email{wjfustc@zjnu.cn}
\address{Department of Mathematics \& Statistics, Memorial University, NL A1C 5S7, Canada}
         \email{jxiao@mun.ca}
\thanks{JW was in part supported the National Natural Science Foundation of China (No.11001246, No. 11101139) and China Scholarship Council.}
\thanks{JX was in part supported by NSERC of Canada and URP of Memorial University.}


\date{}



\begin{abstract}
This article is devoted to not only characterizing the first and second preduals of the analytic Campanato spaces ($\mathcal{CA}_p)$ on the unit disk, but also investigating boundedness of three operators: superposition ($\mathsf{S}^\phi$); backward shift ($\mathsf{S_b}$); Schwarzian derivative ($\mathsf{S}$), acting on $\mathcal{CA}_p$.
\end{abstract}
\maketitle

\tableofcontents

\section{Introduction}\label{s1}
\setcounter{equation}{0}

From now on, $\mathbb D$ and $\mathbb T$ respectively represent the unit disk and the unit circle in the finite complex plane $\mathbb C$. For $p\in (-\infty,\infty)$, $\mathcal{CA}_p$ denotes the Campanato space of all analytic functions $f: \mathbb D\to\mathbb C$ with radial boundary values $f$ on $\mathbb T$ obeying
$$
\|f\|_{\mathcal{CA}_p,*}=\sup_{I\subseteq\mathbb T}\sqrt{|I|^{-p}\int_I|f(\xi)-f_I|^2|d\xi|}<\infty,
$$
where the supremum is taken over all sub-arcs $I\subseteq\mathbb T$ with $|I|$ being their arc-lengths, and
\begin{equation*}
\left\{\begin{array}{ll}|d\xi|=|de^{i\theta}|=d\theta\\
|I|=(2\pi)^{-1}\int_{I}|d\xi| \\
f_I=(2\pi|I|)^{-1}\int_I f(\xi)\,|d\xi|
\end{array}\right.
\end{equation*}\\
Obviously, $\|\cdot\|_{\mathcal{CA}_p,*}$ cannot distinguish between any two $\mathcal{CA}_p$ functions differing by a constant, but 
$$
\|f\|_{\mathcal{CA}_p}=|f(0)|+\|\cdot\|_{\mathcal{CA}_p,*}
$$ 
defines a norm so that $\mathcal{CA}_p$ is a Banach space. The following table tells us that how $\mathcal{CA}_p$ looks like (see, e.g. \cite{Xiao2, XiaoX, XiaoY} and their references):
\medskip
\begin{center}
    \begin{tabular}{ | l | p{7cm} |}
    \hline
    Index $p$ & Analytic Campanato Space $\mathcal{CA}_p$\\ \hline
    $p\in (-\infty,0]$ & Analytic Hardy space $\mathcal H^2$ \\ \hline
    $p\in (0,1)$ & Holomorphic Morrey space $\mathcal H^{2,p}$\\ \hline
    $p=1$ & Analytic John-Nirenberg space $\mathcal{BMOA}$ \\ \hline
    $p\in (1,3]$ & Analytic Lipschitz space $\mathcal A_{\frac{p-1}{2}}$\\ \hline
    $p\in (3,\infty)$ & Complex constant space $\mathbb C$\\ \hline
\end{tabular}
\end{center}
\medskip

\noindent Similarly, the little (or vanishing) analytic Campanato space $\mathcal{CA}_{0,p}$ consists of all functions in $\mathcal{CA}_{p}$ satisfying 
$$
\lim_{|I|\rightarrow 0}{|I|^{-p}\int_{I}|f(\xi)-f_{I}|^{2}|d\xi|}=0.
$$ 
In particular, one has
\begin{center}
    \begin{tabular}{ | l | p{7cm} |}
    \hline
    Index $p$ & Little Analytic Campanato Space $\mathcal{CA}_{0,p}$\\ \hline
    $p\in (-\infty,0]$ &  Analytic Hardy space $\mathcal H^2$ \\ \hline
    $p\in (0,1)$ & Little Holomorphic Morrey space $\mathcal H_0^{2,p}$\\ \hline
    $p=1$ & Analytic Sarason space $\mathcal{VMOA}$ \\ \hline
    $p\in (1,3]$ & Little Analytic Lipschitz space $\mathcal A_{0,\frac{p-1}{2}}$\\ \hline
    $p\in (3,\infty)$ & Complex constant space $\mathbb C$\\ \hline
\end{tabular}
\end{center}
\medskip

According to \cite{Xiao2, XiaoX, Zhu}, if 
$$
\begin{cases}
(p,a,z)\in (0,2)\times\mathbb D\times\mathbb D;\\
\sigma_{a}(z)=\frac{a-z}{1-\bar{a}z};\\
E(f,a)=(1-|a|^{2})^{1-p}\int_{\mathbb{D}}|f'(z)|^{2}(1-|\sigma_{a}(z)|)dm(z).
\end{cases}
$$ 
then not only the following are equivalent:
\begin{itemize}

\item $f\in\mathcal{CA}_{p}$;

\item $|f^{'}(z)|^{2}(1-|z|^{2})dm(z)$ is a bounded $p$-Carleson measure;

\item $|f^{'}(z)|^{2}(1-|\sigma_{a}(z)|)dm(z)$ is a bounded $p$-Carleson measure;

\item $\||f\||_{\mathcal{CA}_{p},*}=\sup\limits_{a\in\mathbb{D}}[E(f,a)]^{\frac{1}{2}}<\infty,$ 
\end{itemize}
but also the following are equivalent:
\begin{itemize}

\item $f\in\mathcal{CA}_{0,p} $;

\item $|f^{'}(z)|^{2}(1-|z|^{2})dm(z)$ is a compact $p$-Carleson measure;

\item $|f^{'}(z)|^{2}(1-|\sigma_{a}(z)|)dm(z)$ is a compact $p$-Carleson measure;

\item $\||f\||_{\mathcal{CA}_{p},*}<\infty\ \ \&\ \ \lim_{|a|\rightarrow 1}E(f,a)=0.$
\end{itemize}
In the above and below, we say that a nonnegative measure $\mu$ on $\mathbb{D}$ is a bounded or compact $(0,\infty)\ni p$-Carleson measure provided 
$$
\|\mu\|_{p}=\sup_{I\subseteq\mathbb{T}}\frac{\mu(S(I))} {|I|^{p}}<\infty\quad\hbox{or}\quad\lim_{I\subseteq\mathbb{T},\ |I|\rightarrow 0}\frac{\mu(S(I))}{|I|^{p}}=0,
$$
where for any arc $I\subseteq\mathbb{T}$ one sets $S(I)=\{z=re^{i\theta}\in\mathbb{D}:\,1-|I|\leq r<1,\,e^{i\theta}\in I\}.$ 

Continuing essentially from \cite[Chapter 3]{Xiao2, AX}, \cite{XiaoX, XiaoY} and \cite{ CascanteFO, LiLL, WuX}, in this paper we study two predualities and three operators associated to the analytic Campanato spaces. More precisely, in \S\ref{s2} we use \S\ref{s21} - the Choquet integrals and quadratic tent spaces to discover \S\ref{s22} - the predual space of $\mathcal{CA}_{p}$ and \S\ref{s23} - the dual space of $\mathcal{CA}_{0,p}$. And, in \S\ref{s3} we discuss: \S\ref{s31} - when the superposition $\mathsf{S}^\phi$ is bounded on $\mathcal{CA}_{p}$ and $\mathcal{CA}_{0,p}$; \S\ref{s32} - the boundedness of the backward shift $\mathsf{S_b}$ on both $\mathcal{CA}_{p}$ and $\mathcal{CA}_{0,p}$; \S\ref{s33} - the behavior of the Schwarzian derivative $\mathsf{S}(f)$ of a univalent function $f$ on $\mathbb D$ whenever $\log f'$ is in $\mathcal{CA}_{p}$ or $\mathcal{CA}_{0,p}$.

\medskip
{\it Notation}: In this note, we will use ${\mathsf X}\lesssim{\mathsf Y}$ or
${\mathsf X}\gtrsim{\mathsf Y}$ to express ${\mathsf
X}\le\kappa{\mathsf Y}$ or ${\mathsf X}\ge \kappa{\mathsf Y}$ for some constant $\kappa>0$. Moreover, ${\mathsf X}\approx{\mathsf Y}$ means ${\mathsf X}\lesssim{\mathsf Y}$ and
${\mathsf X}\gtrsim{\mathsf Y}$. In addition, $dm(z)=dxdy$ stands for two dimensional Lebesgue area measure.

\section{Two predualities}\label{s2}
\setcounter{equation}{0}
\subsection{Choquet integrals and tent spaces}\label{s21}
The $(0,1]\ni p$-dimensional capacity of $E\subseteq\partial{\mathbb{D}}$ is defined by 
$$
\Lambda_{p}^{(\infty)}(E)=\inf\left\{\sum_{j=1}^{\infty}|I_{j}|^{p}:E\subseteq\bigcup\limits_{j=1}^{\infty}I_{j}\right\},
$$
where the infimum is taken over all coverings of $E$ by countable families of open arcs $I_{j}\subseteq{\mathbb{T}}$. According to \cite{Adams}, $\Lambda_{p}^{(\infty)}$ is a monotone, countably subadditive set function on the class of all subsets of ${\mathbb{T}}$ which vanishes on the empty set, and the Choquet integral of a nonnegative function $f$ on ${\mathbb{T}}$ against $\Lambda_{p}^{(\infty)}$ is defined by 
$$
\int_{{\mathbb{T}}}fd\Lambda_{p}^{(\infty)}=
\int_{0}^{\infty}\Lambda_{p}^{(\infty)}\big(\{\xi\in
{\mathbb{T}}:f(\xi)>t \}\big)\,dt.
$$
Following \cite[Chapter 4]{Xiao2}, let 
$$
N(\omega)(\xi)=
\sup\{\omega(z):|z-\xi|<1-|z|^{2}\}\quad\forall\quad\xi\in\mathbb{T}.
$$
be the nontangential maximal function of $\omega$, and define
$$
\|\omega\|_{LN^{1}(\Lambda_{\infty}^{p})}=\int_{{\mathbb{T}}}N(\omega) d\Lambda_{p}^{(\infty)}.
$$

\begin{definition}
\label{d21}
Let $p\in (0,1]$. 
\begin{itemize}
\item The space $\mathcal{T}_{p}^{\infty}$ consists of all Lebesgue measurable functions $f$ on $\mathbb{D}$ with
$$
\|f\|_{\mathcal{T}_{p}^{\infty}}=\sup_{{I\subseteq \mathbb{T}}}\left(|I|^{-p}\int_{\mathbb{T}(I)}|f(z)|^{2}(1-|z|^{2})dm(z)\right)^{\frac{1}{2}}<\infty,
$$
where the supremun runs over all open subarcs $I$ of $\mathbb{T}$ and 
$$
\mathbb{T}(E)=\{re^{i\theta}\in \mathbb{D}: \mathrm{dist}(e^{i\theta},\mathbb{T}\setminus E)>1-r\}
$$
is the tent over the set $E\subset \mathbb{T}.$

\item The space $\mathcal{T}_{p}^{1}$ consists of all Lebesgue measurable functions $f$ on $\mathbb{D}$ with
$$ \|f\|_{\mathcal{T}_{p}^{1}}=\inf_{{\omega}}\left(|I|^{-p}\int_{\mathbb{D}}|f(z)|^{2}(\omega(z))^{-1}(1-|z|^{2})^{-1}dm(z)\right)^{\frac{1}{2}}<\infty, 
$$
where the above infimun is taken over all nonnegative functions $\omega$ on $\mathbb{D}$ with $
\|\omega\|_{LN^{1}(\Lambda_{\infty}^{p})}\leq 1.$

\item A function $a$ on $\mathbb{D}$ is called a $\mathcal{T}_{p}^{1}$-atom if there exists a subarc $I$ of $\mathbb{T}$ such that $a$ is supported in the tent $\mathbb{T}(I)$ and satisfies 
$$\int_{\mathbb{T}(I)}|a(z)|^{2}(1-|z|^{2})^{-1}dm(z)\leq|I|^{-p}.$$
\end{itemize}
\end{definition}

\begin{lemma}
\label{l21}
Let $p\in (0,1]$. 

\begin{itemize}

\item If $\sum\limits_{j=1}^{\infty}||f_{j}||_{{\mathcal{T}_{p}^{1}}}<\infty$, then 
$f=\sum\limits_{j=1}^{\infty}f_{j}\in{\mathcal{T}_{p}^{1}}$ with $\leq2\sum\limits_{j=1}^{\infty}||f_{j}||_{{\mathcal{T}_{p}^{1}}}.$

\item $f\in \mathcal{T}_{p}^{1}$ if and only if there is a sequence of $\mathcal{T}_{p}^{1}$-atom $\{a_{j}\}$ and an $l^{1}$-sequence  $\{\lambda_{j}\}$ such that 
$f=\sum\limits_{j=1}^{\infty}\lambda_{j}a_{j}$. Moreover, 
$$
\|f\|_{\mathcal{T}_{p}^{1}}\approx \||f\||_{\mathcal{T}_{p}^{1}}=\inf\left\{\sum\limits_{j=1}^{\infty}|\lambda_{j}|:\ \ f=\sum\limits_{j=1}^{\infty}\lambda_{j}a_{j}\right\},
$$ 
where the infimum is taken over all possible atomic decompositions of $f\in {\mathcal{T}_{p}^{1}}.$ Consequently,  ${\mathcal{T}_{p}^{1}}$ is a Banach space under the norm $\||f\||_{\mathcal{T}_{p}^{1}}.$

\item $\mathcal{T}_{p}^{\infty}=[{\mathcal{T}_{p}^{1}}]^{*}$ under the pairing $<f,g>=\frac{1}{\pi}\int_{\mathbb{D}}f\bar{g}dm.$

\end{itemize}

\end{lemma}
\begin{proof} This follows easily from a slight modification of \cite[Lemma 4.3.1 \& Theorem 4.3.2]{Xiao2}.
\end{proof}
\subsection{First predual}\label{s22}

When $1<p\leq3$, Duren, Romberg and shierds \cite{DurenRS} gave the predual space of the analytic Lipschitz $\mathcal{A}^{\frac{p-1}{2}}$ is the Hardy space $\mathcal{H}^{\frac{2}{1+p}}.$  For $p=1$, Fefferman \cite{FeffermanS} established the well-known result $ (\mathcal{H}^1)^{*}=\mathcal{BMOA}$. For $p=0$, $(\mathcal{H}^2)^{*}=\mathcal{H}^2$. For $p\in (0,1)$, note that $\mathcal{BMOA}\subset\mathcal{H}^{2,p} \subset\mathcal{H}^2$, hence the predual of the analytic Morrey space should be an analytic function space between the analytic Hardy spaces $\mathcal{H}^{2}$ and $\mathcal{H}^{1}$. To work out this predual space, we need the following lemma.

\begin{lemma}\cite{WuX}
\label{l22}
For $p,\eta\in(0,2),\,a>
\frac{2-\eta}{2},\,b>\frac{1+\eta}{2}$, and a Lebesgue measurable function $f$ on $\mathbb{D}$, let
$$
\textbf{T}_{a,b}f(z)=\int_{\mathbb{D}}\frac{(1-|w|^{2})^{b-1}}{|1-\bar{w}z|^{a+b}}f(w)dm(w),\,\,z\in\mathbb{D}.
$$
If $|f(z)|^{2}(1-|z|^{2})^{\eta}dm(z)$ is a $p$-Carleson measure, then $|\textbf{T}_{a,b}f(z)|^{2}(1-|z|^{2})^{2a+\eta-2}dm(z)$ is also a $p$-Carleson measure.
\end{lemma}

Below is a description of the first predual space of the analytic Morrey space.

\begin{theorem}
\label{t21}
For $p\in (0,1]$ let $\mathcal{BA}_p$ be the class of all analytic functions $f$ on $\mathbb{D}$ satisfying $f(0)=0$ and 
$$
\|f\|_{\mathcal{BA}_p}=\inf_{\omega}\left(\int_{\mathbb{D}}|f'(z)|^{2}(\omega(z))^{-1}(1-|z|^{2})dm(z)\right)^{\frac{1}{2}}<\infty,
$$
where the infimun is taken over all nonnegative functions $\omega$ on $\mathbb{D}$ with $\|\omega\||_{LN^{1}(\Lambda_{\infty}^{p})}\leq 1.$ Then $\mathcal{CA}_{p}$ is isomorphic to the dual of $\mathcal{BA}_p$  under the pairing 
$$
\langle f,g\rangle=\int_{\mathbb{T}}f(\xi){\bar{g}(\xi)}\frac{|d\xi|}{2\pi}\quad\forall\quad (f,g)\in\mathcal{BA}_p\times\mathcal{CA}_{p}.
$$ 
That is, $\mathcal{BA}_p^{*}=\mathcal{CA}_{p}.$ Moreover, 
$$\|f\|_{\mathcal{BA}_p}\approx
\sup\{|\langle{f},{g}\rangle|:\ \ g\in\mathcal{CA}_{p}\ \ \&\ \ \|g\|_{\mathcal{CA}_{p}}\leq 1\}\quad\forall\quad f\in\mathcal{BA}_p.
$$
\end{theorem}
\begin{proof} On the one hand, assume that $f\in\mathcal{BA}_p$ and
$g\in{\mathcal{CA}_{p}}$. According to the function-theoretic characterization of $\mathcal{CA}_{p}$ stated in \S\ref{s1} we have 
$$
d\mu(z)=|g'(z)|^{2}(1-|z|^{2})dm(z)
$$ 
is a $p$-Carleson measure on $\mathbb{D}$; that is
$$
\|\mu\|_{p}=\sup_{I\subseteq{\mathbb{T}}}\mu(S(I))|I|^{-P}\lesssim\|g\|_{\mathcal{CA}_{p}}^{2}.
$$ 
If $\omega$ is a positive function on $\mathbb{D}$ satisfying 
$$
\int_{\mathbb{T}}N(\omega)(\xi)d\Lambda_{p}^{\infty}(\xi)\leq 1,
$$ 
then, by the Hardy-Littlewood identity and the Cauchy-Schwarz inequality we obtain 
\begin{align*}
&\Big|\frac{1}{2\pi}\int_{\mathbb{T}}f(e^{i\theta})\bar{g}(e^{i\theta})d\theta\Big|\\
&=\Big|f(0)\bar{g}(0)+\frac{2}{\pi}\int_{\mathbb{D}}f'(z)\bar{g}'(z)\log\frac{1}{|z|}dm(z)\Big|
\\
&\lesssim\int_{\mathbb{D}}|f'(z)||g'(z)|(1-|z|^{2})dm(z)\\
&\lesssim 
\left(\int_{\mathbb{D}}|g'(z)|^{2}\omega(z)(1-|z|^{2})dm(z)\right)^{\frac{1}{2}}\left(\int_{\mathbb{D}}|f'(z)|^{2}(\omega(z))^{-1}(1-|z|^{2})dm(z)\right)^{\frac{1}{2}}\\
&\lesssim
\left(\int_{\mathbb{D}}(\omega(z))d\mu(z)\right)^{\frac{1}{2}}\left(\int_{\mathbb{D}}|f'(z)|^{2}(\omega(z))^{-1}(1-|z|^{2})dm(z)\right)^{\frac{1}{2}}\\
&\lesssim\|\mu\|_{p}^{\frac{1}{2}}\left(\int_{\mathbb{D}}|f'(z)|^{2}(\omega(z))^{-1}(1-|z|^{2})dm(z)\right)^{\frac{1}{2}}\\
&\lesssim\|g\|_{\mathcal{CA}_{p}}\left(\int_{\mathbb{D}}|f'(z)|^{2}(\omega(z))^{-1}(1-|z|^{2})dm(z)\right)^{\frac{1}{2}}
\end{align*}
Hence, 
$$
|\langle f,g\rangle|\lesssim\|g\|_{\mathcal{CA}_{p}}\|f\|_{\mathcal{BA}_p},
$$
namely, $\mathcal{CA}_{p}\subseteq\mathcal{BA}_p^{*}$.

On the other hand, suppose $L\in\mathcal{BA}_p^{*}$. If
$$
D(f)(z)=f'(z)(1-|z|^{2})\quad\forall\quad z\in\mathbb{D},
$$
then ${D}$ is an isometric map from $\mathcal{BA}_p$ into $\mathcal{T}_{p}^{1}.$ Since $\mathcal{T}_{p}^{1}$ is a Banach space under $\||\cdot\||_{\mathcal{T}_{p}^{1}}$, it follows from the Hahn-Banach Theorem and Lemma \ref{l21} that one can select a function 
$h\in{\mathcal{T}_{p}^{\infty}}$ such that
$$
L(f)=\int_{\mathbb{D}}f'(z)(1-|z|^{2})\bar{h}(z)dm(z)\quad\forall\quad f\in\mathcal{BA}_p.
$$
Since 
$$
\int_{\mathbb{T}}N(\omega)(\xi)d\Lambda_{p}^{\infty}(\xi)\leq 1\Longrightarrow (1-|z|^{2})^{p}\omega(z)\lesssim1\quad\forall\quad z\in\mathbb D,
$$
one has
$$
\int_{\mathbb{D}}|f'(z)|(1-|z|^{2})^{2}dm(z)\lesssim\int_{\mathbb{D}}|f'(z)|(1-|z|^{2})^{1+p}dm(z)\lesssim\int_{\mathbb{D}}|f'(z)|(\omega(z))^{-1}(1-|z|^{2})dm(z)
$$
holds for any $\omega$ used in $\mathcal{BA}_p$. Thus we utilize the reproducing formula in \cite[p.120]{Rudin} or \cite[p.81]{Zhu} to achieve
$$
f'(z)=\frac{2}{\pi}\int_{\mathbb{D}}\frac{f'(w)(1-|w|^{2})}{(1-\bar{w}z)^{3}}dm(w)\quad \forall\quad f\in\mathcal{BA}_p,
$$
thereby finding

\begin{align*}
&\textbf{L}(f)\\
&=\frac{1}{\pi}\int_{\mathbb{D}} f'(z)(1-|z|^{2})\bar{h}(z)dm(z)\\
&=\frac{1}{\pi}\int_{\mathbb{D}}f'(w)
\left(\frac{2}{\pi}\int_{\mathbb{D}}\frac{\bar{h}(z)(1-|z|^{2})}{(1-\bar{w}z)^{3}}dm(z)\right)(1-|w|^{2})dm(w)
\\
&=\frac{1}{\pi}\int_{\mathbb{D}}f'(w)\bar{g}'(w)(1-|w|^{2})dm(w),
\end{align*}
where $$g(w)=\frac{2}{\pi}\int_{0}^{w}\left(\int_{\mathbb{D}}\frac{{h}(z)(1-|z|^{2})}{(1-u\bar{z})^{3}}dm(z)\right)du.$$
Note that 
$$
g(0)=0\ \ \&\ \ \Big|\int_{\mathbb{D}}f'(w)\bar{g}'(w)(1-|w|^{2})dm(w)\Big|\approx\Big|\int_{\mathbb{D}}f'(w)\bar{g}'(w)\log(\frac{1}{|w|})dm(w)\Big|.
$$ 
In terms of the Hardy-Littlewood identity, we get $|\textbf{L}(f)|\approx|\langle f,g\rangle|$. Note also that 
$$
d\mu(z)=|h(z)|^{2}(1-|z|^{2})dm(z)
$$ 
is a $p$-Carleson measure on $\mathbb{D}$. So, Lemma \ref{l22} is employed to
derive that 
$$
|g'(w)|^{2}(1-|w|^{2})dm(z)
$$ 
is also a $p$-Carleson measure on $\mathbb{D}$, and then $g\in\mathcal{CA}_{p}.$ Therefore, $\mathcal{BA}_p^{*}=\mathcal{CA}_{p}.$ Furthermore, applying a consequence of the Hahn-Banach Extension Theorem (see \cite[p.48]{Devito}) to $\mathcal{BA}_p$, we see that if $f\in\mathcal{BA}_p$ is a nonzero then there exists $\textbf{L}\in\mathcal{BA}_p^{*}=\mathcal{CA}_{p}$ such that 
$$
\|\textbf{L}\|=1\quad\&\quad
\|f\|_{\mathcal{BA}_p}=\textbf{L}(f)=\langle f,g\rangle.
$$ 
With the help of the foregoing argument, we can find a function $g\in\mathcal{CA}_{p}$ such that 
$$
\|g\|_{\mathcal{CA}_{p}}\leq 1\quad\&\quad
\textbf{L}(f)=\langle f,g\rangle\quad\forall\quad f\in\mathcal{BA}_p.
$$ 
This clearly implies that 
 $$
 \|f\|_{\mathcal{BA}_p}\approx
\sup\{|\langle{f},{g}\rangle|:\ \ g\in\mathcal{CA}_{p}\ \&\ \|g\|_{\mathcal{CA}_{p}}\leq 1\}\quad\forall\quad f\in\mathcal{BA}_p.
$$
\end{proof}

\subsection{Second predual}\label{s23}
 It is well-known that 
 $$
 \begin{cases}[\mathcal{H}^{2}]^{**}=\mathcal{H}^{2};\\
 [\mathcal{VMOA}]^{**} =\mathcal{BMOA};\\
  [\mathcal{A}_{0,\frac{p-1}{2}}]^{**}=\mathcal{A}_{\frac{p-1}{2}}\quad\forall\quad p\in (1,3).
  \end{cases}
  $$
  So, it remains to see whether the above identification can be extended to the analytic Morrey space. To do so, we need two lemmas.
  \begin{lemma}\label{l23}
  For $(p,r)\in(0,1]\times(0,1)$ and $f\in\mathcal{BA}_p$ let $f_{r}(z)=f(rz)$. Then $$
 \lim_{r\rightarrow 1}\|f_{r}-f\|_{\mathcal{BA}_p}=0\quad\&\quad \|f_{r}\|_{\mathcal{BA}_p}\leq\|f\|_{\mathcal{BA}_p}.
 $$ 
 Thus the polynomials are dense in $\mathcal{BA}_p$. 
 \end{lemma}
 \begin{proof} Choosing $\omega(z)=1$ in the definition of $\mathcal{BA}_p$, one gets that any bounded analytic function $f$ with $f(0)=0$ must be in $\mathcal{BA}_p$ and hence $f_{r}\in\mathcal{BA}_p$ for any $0<r<1.$ Suppose now
 $f\in\mathcal{BA}_p$. An  
 application of Poisson's formula to $f_{_{r}}$ gives 
 \begin{equation*}
 f_{r}(z)=\frac{1}{2\pi}\int_{\mathbb{T}}f(z\zeta)\frac{1-r^{2}}{|1-r\bar{\zeta}|^{2}}|d\zeta|.
 \end{equation*}
 Derivating both sides of the above equality with respect to $z$ and using Minkowski's inequality, we have 
 \begin{equation}\label{eee}
 |f_{r}'(z)|\leq \left(\frac{1}{2\pi}\int_{\mathbb{T}}|f'(z\zeta)|^{2}\frac{1-r^{2}}{|1-r\bar{\zeta}|^{2}}|d\zeta|\right)^{\frac{1}{2}}.
 \end{equation}
 For any $\epsilon>0$, by the definition of $\mathcal{BA}_p$, there is a nonnegative $\omega_{\epsilon}$ on $\mathbb{D}$ such that 
 \begin{equation*}
 \left(\int_{\mathbb{D}}|f'(z)|^{2}(\omega_{\epsilon}(z))^{-1}(1-|z|^{2})dm(z)\right)^{\frac{1}{2}}<\|f\|_{\mathcal{BA}_p}+\epsilon
 \end{equation*}
 According to the rotation invariance of $\omega_{\epsilon}$, one has that for any $\zeta\in\mathbb{T}$  
 \begin{equation}\label{eeee}
 \left(\int_{\mathbb{D}}|f'(z\zeta)|^{2}(\omega_{\epsilon}(z))^{-1}(1-|z|^{2})dm(z)\right)^{\frac{1}{2}}<||f||_{\mathcal{BA}_p}+\epsilon
 \end{equation}
 Using the inequalities (\ref{eee})-(\ref{eeee}) and Fubini's theorem, we obtain
 \begin{align*}
&\|f_{r}\|^{2}_{\mathcal{BA}_p}\\
&\leq\int_{\mathbb{D}}|f_{r}'(z)|^{2}(\omega_{\epsilon}(z))^{-1}(1-|z|^{2})dm(z)  
 \\
 &\leq\frac{1}{2\pi}\int_{\mathbb{T}}\left(\int_{\mathbb{D}}|f'(z\zeta)|^{2}(\omega_{\epsilon}(z))^{-1}(1-|z|^{2})dm(z)\right)\frac{1-r^{2}}{|1-r\bar{\zeta}|^{2}}|d\zeta|
 \\
 &\leq\frac{1}{2\pi}\int_{\mathbb{T}}(||f||_{\mathcal{BA}_p}+\epsilon)^{2}\frac{1-r^{2}}{|1-r\bar{\zeta}|^{2}}|d\zeta|\\
 &=(||f||_{\mathcal{BA}_p}+\epsilon)^{2}.
 \end{align*}
 Letting $\epsilon\rightarrow 0$ in the above estimates, one obtains
 $\|f_{r}\|_{\mathcal{BA}_p}\leq\|f\|_{\mathcal{BA}_p}.$

 In the sequel we prove
 $\lim\limits_{r\rightarrow 1}||f_{r}-f||_{\mathcal{BA}_p}=0$. 
 For any $0<\delta<1$, we have
 \begin{align*}
 &\int_{\mathbb{D}}|f_{r}'(z)-f'(z)|^{2}
 (\omega(z))^{-1}(1-|z|^{2})dm(z)\\
 &=\int_{|z|\leq\delta}|f_{r}'(z)-f'(z)|^{2}
 (\omega(z))^{-1}(1-|z|^{2})dm(z)\\
 &+\int_{1-\delta<|z|\leq 1}|f_{r}'(z)-f'(z)|^{2}
 (\omega(z))^{-1}(1-|z|^{2})dm(z)\\
 &=Int_{1}+Int_{2}.\\
 \end{align*}
 Note that $f,\,f_r\in\mathcal{BA}_p$. So, for any $\epsilon>0$ there exists $\eta>0$ such that
 $$1-\eta<\delta<1\Longrightarrow Int_{2}<\epsilon^{2}.
 $$
 Now, fixing some $\delta$ and noticing that $f_r$ is uniformly convergent to $f$ on the compact set $\{z\in\mathbb{C}:\,|z|\leq\delta\}$, one gains a number $\lambda>0$ such that  
 $$
 1-\lambda<r<1\Longrightarrow Int_{1}<\epsilon^{2}.
 $$ 
 Thus, putting all together gives  
 $\lim_{r\rightarrow 1}\|f_{r}-f\|_{\mathcal{BA}_p}=0$.
 \end{proof}
 
 A modification of the techniques used in \cite{WX} produces the following density result for $\mathcal{CA}_{0,p}$.
 
 \begin{lemma} \label{l24} For $(p,r)\in(0,1]\times(0,1)$ let $f\in{\mathcal{CA}_{p}}$ with $f_{r}(z)=f(rz)$. Then the following are equivalent:
 \begin{itemize}
 \item $f\in{\mathcal{CA}_{0,p}}$, i.e., $\lim_{I\subseteq\mathbb{T}, \ |I|\to 0}\sqrt{|I|^{-p}\int_{I}|f(\xi)-f_{I}|^{2}|d\xi|}=0$;
 
 \item $\lim_{r\to 1}||f_{r}-f||_{\mathcal{CA}_{p}}=0$;
 
 \item $f$ belongs to the closure of all polynomials in the norm $\|\cdot\|_{\mathcal{CA}_{p}};$
 
 \item For any $\epsilon>0$ there is a $g\in{\mathcal{CA}_{0,p}}$ such that $\|g-f\|_{\mathcal{CA}_{p}}<\epsilon.$
 \end{itemize}
 \end{lemma}    
 
 \begin{proof} It is enough to verify that 
 $$
 f\in{\mathcal{CA}_{0,p}}\Longrightarrow\lim_{r\to 1}\|f_{r}-f\|_{\mathcal{CA}_{p}}=0.
 $$
 Suppose now $f\in{\mathcal{CA}_{0,p}}$. Then
 $$
 \lim_{|a|\rightarrow 1}E(f_r-f, a)=0
 $$ 
 holds for any fixed $r\in(0,1)$. Meanwhile, 
 $$
 \lim_{r\to 1}\sup_{|a|\leq\delta}E(f_r-f,a)=0.
 $$
 Also, it is not hard to establish 
 
\begin{align*}
&\||f_r\||_{{\mathcal{CA}_{p}},*}\\
&=\sup_{a\in\mathbb{D}}[E(f_{r},a)]^{\frac{1}{2}}\\
&=\sup_{a\in\mathbb{D}}(1-|a|^{2})^{\frac{1-p}{2}}\left(\int_{\mathbb{D}}|(f_{r})'(z)|^{2}(1-|\sigma_{a}(z)|)dm(z)\right)^{\frac{1}{2}}\\
&\leq\frac{1}{2\pi}\int_{\mathbb{T}}\sup_{a\in\mathbb{D}}(1-|a|^{2})^{\frac{1-p}{2}}\left(\int_{\mathbb{D}}|f'(z\zeta)|^{2}(1-|\sigma_{a}(z)|)dm(z)\right)^{\frac{1}{2}}\frac{1-r^{2}}{|1-r\bar{\zeta}|^{2}}|d\zeta|\\
 &\leq[E(f,a)]^{\frac{1}{2}}\\
 &=\||f\||_{{\mathcal{CA}_{p}},*}.
 \end{align*}
Summing up, one finds that 
 $$
 \lim_{r\to 1}\||f_{r}-f\||_{{\mathcal{CA}_{p}},*}=0\quad\&\quad \lim_{r\to 1}\|f_{r}-f\|_{{\mathcal{CA}_{p}}}=0,
 $$
 as desired.
 \end{proof}
 
 Below is the second preduality for the analytic Morrey space.
 
 \begin{theorem}
 \label{t22}
 Let $p\in (0,1]$. Then $\mathcal{BA}_p$ is isomorphic to the dual of ${\mathcal{CA}_{0,p}}$ under the pairing $$
 \langle f,g\rangle=\int_{\mathbb{T}}f(\xi)\bar{g}(\xi)\frac{|d\xi|}{2\pi}\quad\forall\quad (f, g)\in{\mathcal{CA}_{0,p}}\times
 \mathcal{BA}_p.
 $$
  That is, $[{\mathcal{CA}_{0,p}}]^{*}=\mathcal{BA}_p.$ Thus
  $[{\mathcal{CA}_{0,p}}]^{**}={\mathcal{CA}_{p}}.$
 \end{theorem}
 \begin{proof}
On the one hand, for $g\in\mathcal{BA}_p$ let
 $$ \mathbf{S}_{g}(f)=\bar{\mathbf{T}}_f(g)=\int_{\mathbb{T}}f(\xi)\bar{g}(\xi)\frac{|d\xi|}{2\pi}\quad\forall\quad f\in{\mathcal{CA}_{0,p}}.
 $$ 
 Then $\mathbf{S}_{g}$ is linear. Also
 $$
 |\mathbf{S}_{g}(f)|\lesssim \|f\|_{\mathcal{CA}_{p}}
 \|g\|_{\mathcal{BA}_p}.
 $$
 Thus 
 $$
 \mathbf{S}_{g}\in [{\mathcal{CA}_{0,p}}]^{*}\quad\&\quad
 \|\mathbf{S}_{g}\|\lesssim\|g\|_{\mathcal{BA}_p}.
 $$
 Conversely, suppose $\mathbf{S}\in[{\mathcal{CA}_{0,p}}]^{*}$. For each $n\in\{0,1,2,...\}$ let 
 $$\phi_{n}(z)=z^{n}\quad\forall\quad z\in\mathbb{D}.
 $$
 Then $\phi_{n}\in{\mathcal{CA}_{0,p}}$ for all $n$. According to \cite[Lemma 1]{XiaoX}, an elementary calculation shows 
 $$
 1\leq|||\phi_{_{_{n}}}|||_{\mathcal{CA}_{p},*}=
 \sup_{a\in\mathbb{D}}(1-|a|^{2})^{\frac{1-p}{2}}\|\phi_{n}\circ\sigma_{a}-\phi_{n}(a)\|_{\mathcal{H}^{2}}\leq 2,
 $$
 where $\sigma_{a}(z)=\frac{a-z}{1-\bar{a}z}.$
 Set 
 $$b_{n}=\bar{\mathbf{S}}(\phi_{n})\quad\forall\quad n=0,1,2,...$$
 Then we see that
 $$
 |b_{n}|=|\mathbf{S}(\phi_{n})|\leq \|\mathbf{S}\|\,\|\phi\|_{\mathcal{CA}_{p}}\lesssim\|\mathbf{S}\|.
 $$
 It follows that the power series $\sum\limits_{n=0}^{\infty}b_{n}z^{n}$ has its radius of convergence greater than or equal to $1$.
 Define
 $$
 g(z)=\sum\limits_{n=0}^{\infty}b_{n}z^{n}\quad \forall\quad z\in\mathbb{D}.
 $$
 Then $g$ is analytic in $\mathbb{D}$. For $0<r<1$, set
 $$
 g_{r}(z)=g(rz)\quad\forall\quad z\in\mathbb{D}.
 $$
 We are going to show that 
 \begin{equation}\label{1eee}
 g\in\mathcal{BA}_p\quad\hbox{with}\quad\|g\|_{\mathcal{BA}_p}\lesssim\|\mathbf{S}\|
 \end{equation}
 and that 
 \begin{equation}\label{2eee}
 \mathbf{S}=\mathbf{S}_{g}.
 \end{equation}
 In terms of Lemma \ref{l23}, (\ref{1eee}) is equivalent to 
 $$
 \|g_{r}\|_{\mathcal{BA}_p}\lesssim\|\mathbf{S}\|\quad\forall\quad r\in (0,1).
 $$
 Theorem \ref{t21} shows that 
 $$
 \|\mathbf{T}_{f}\|\approx\|f\|_{\mathcal{CA}_{p}}.
 $$
 Taking $f\in\mathcal{BA}_p$ and setting $f(z)=\sum\limits_{n=0}^{\infty}a_{n}z^{n}$, one obtains 
 $$
 f_{r}(z)=f(rz)=\sum\limits_{n=0}^{\infty}a_{n}r^{n}z^{n}\quad\forall\quad z\in\mathbb{D}.
 $$
 Clearly, 
 $$
 \sum\limits_{k=0}^{n}a_{k}z^{k}=\sum\limits_{k=0}^{n}a_{k}\phi_{k}\rightarrow f_{r}\,\,\hbox{as}\,\, n\rightarrow \infty
 $$
 uniformly in $\bar{\mathbb{D}}$ and, hence, in $\mathcal{BA}_p$, which, since $\mathbf{S}\in[{\mathcal{CA}_{0,p}}]^{*}$, implies
 $$
 \mathbf{S}(\sum\limits_{k=0}^{n}a_{k}r^{k}\phi_{k})
 =\sum\limits_{k=0}^{n}a_{k}\bar{b}_{k}r^{k}
 \rightarrow \mathbf{S}(f_{r})\,\,\hbox{as}\,\, n\rightarrow \infty.
 $$
 That is,
 \begin{equation}\label{3eee}
 \sum\limits_{k=0}^{\infty}a_{k}\bar{b}_{k}r^{k}=
 \mathbf{S}(f_{r}).
 \end{equation}
 Now (\ref{3eee}) can be written as 
 \begin{equation*}
 \mathbf{S}(f_{r})=\mathbf{T}_{g_{r}}(f).
 \end{equation*}
 Using Lemmas \ref{l23}-\ref{l24} we have $\|f_{r}\|_{\mathcal{CA}_{p}}\lesssim
 \|f\|_{\mathcal{CA}_{p}}
 $. Thus,
 $$|\mathbf{T}_{f}(g_{r})|=|\mathbf{T}_{g_{r}}(f)|
 =|\mathbf{S}(f_{r})|\leq\|\mathbf{S}\|\|f_{r}\|_{\mathcal{CA}_{p}}\lesssim\|\mathbf{S}\|f\|_{\mathcal{CA}_{p}}.
 $$
 From Theorem \ref{t21} it follows that
 $$
 \|g_{r}\|_{\mathcal{BA}_p}=
 \sup\limits_{f\in\mathcal{CA}_{p}\setminus\{0\}}
 \frac{|\mathbf{T}_{f}(g_{r})|}{\|\mathbf{T}_{f}\|}\lesssim \sup\limits_{f\in\mathcal{CA}_{p}\setminus\{0\}}\frac{|\mathbf{T}_{f}(g_{r})|}{\|f\|_{\mathcal{CA}_{p}}}.
 $$
 Therefore,
 $$
 \|g_{r}\|_{\mathcal{BA}_p}\lesssim \|\mathbf{S}\|.
 $$
 However, Lemma \ref{l23} tells us that 
 $$
 \|g\|_{\mathcal{BA}_p}\lesssim \|\mathbf{S}\|,
 $$
 which completes the proof of (\ref{1eee}).
 
 Now, if $f\in{\mathcal{CA}_{0,p}}$, according to Lemma \ref{l24} and the continuity of $\mathbf{S}$, one has that
 $$
 \mathbf{S}(f)=\lim\limits_{r\rightarrow 1}\mathbf{S}(f_{r})=\lim\limits_{r\rightarrow 1}\sum\limits_{k=0}^{\infty}\bar{b_{k}}
 a_{k}r^{k}.
 $$ 
 Hence, the proof of (\ref{2eee}) is completed.
  \end{proof}
  
  As an immediate consequence of the second preduality established above, the following covers the corresponding $\mathcal{BMOA}$-result in \cite{AxlerS, CarmonaC, StegengaS} and $\mathrm{Lip}_{\alpha}$-result in \cite{Kalton}. 
 
 \begin{corollary}
 \label{c21}
 Let $p\in[0,3)$ and $f\in{\mathcal{CA}_p}$. Then 
 \begin{equation*}
 \mathrm{dist}(f,\,\,{\mathcal{CA}_{0,p})_{{\mathcal{CA}_p}}=\overline{\lim\limits_{|I|\rightarrow 0}}\left({|I|^{-p}}\int_{I}|f-f_{I}|^{2}|d\xi|\right)^\frac12}.
 \end{equation*}
 \end{corollary}
 \begin{proof} Let
 $$
 \begin{cases}
 X=\mathcal{H}^{2}\setminus\mathbb{C};\\
 Y=\mathcal{H}^{1};\\ 
 L=\Big\{L_{I}:\,L_{I}=\chi_{I}{|I|^{-p}}(f-f_I)\ \ \&\ \ \emptyset\neq I\subseteq\mathbb{T}\,\, \mathrm{is}\,\, \mathrm{an} \,\,\mathrm{arc}\Big
 \},
 \end{cases}
 $$
 where $\chi_{I}$ is the characteristic function of $I$ and $f_I=|I|^{-1}\int_I f(\xi)\,|d\xi|.$ Now, if 
 $$
 \begin{cases}
 f\in\mathcal{CA}_p;\\
 f_n=f*P_{1-\frac{1}{n}}\quad\forall\quad n=1,2,3,...;\\
 P_r(\theta)=\frac{1-r^{2}}{|e^{i\theta}-r|^{2}},
 \end{cases}
 $$
 then an application of the second preduality $[\mathcal{CA}_{0,p}]^{**}=\mathcal{CA}_{p}$ (cf. Theorem \ref{t22}) and
 Lemma \ref{l24} gives that $\mathcal{CA}_p$ and $\mathcal{CA}_{0,\,p}$ satisfy  Assumption A of \cite[Theorem 2.3]{Perfekt}, and hence one has 
 \begin{equation*}
 \mathrm{dist}(f,\,\,{\mathcal{CA}_{p})_{{\mathcal{CA}_p}}=\overline{\lim\limits_{|I|\rightarrow 0}}\left({|I|^{-p}}\int_{I}|f-f_{I}|^{2}|d\xi|\right)^\frac12.}
 \end{equation*}
 \end{proof}
 
\section{Three operators}\label{s3}
\setcounter{equation}{0}

\subsection{Superposition}\label{s31}
Denote by $\mathcal{H}(\mathbb{D})$ the space of analytic functions on $\mathbb{D}$. If $\mathcal X$ and $\mathcal Y$ are two subspaces of $\mathcal{H}(\mathbb{D})$ and $\phi$ is a complex-valued function  $\mathbb{C}$ such that $\phi\circ f\in \mathcal{Y}$ whenever $f\in\mathcal{X}$, we say that $\phi $ acts by superposition from $\mathcal{X}$ into $\mathcal{Y}$. If $\mathcal{X}$ and $\mathcal{Y}$ contain the linear functions, then $\phi$ must be an entire function. The superposition $\mathsf{S}^\phi:\ \mathcal{X}\rightarrow\mathcal{Y}$ with symbol $\phi$ is then defined by $\mathsf{S}^\phi(f)=\phi\circ f$. A basic question is when $\mathsf{S}^\phi$ map $\mathcal{X}$ into $\mathcal{Y}$ continuously? This question has been studied for many distinct pairs $(\mathcal{X},\mathcal{Y})$ - see e.g. \cite{ AlvarezMV, AppellZ, GirelaAM, Xiao2}. In this section, we are interested in the analytic Morrey space and its little one, and have the following result which extends the case of $p\in(1,3]$ in \cite{Xu}.  

\begin{theorem}
\label{t31}
Let $p\in[0,1)$. Then $\mathsf{S}^\phi$ is bounded on $\mathcal{CA}_{p}$ or $\mathcal{CA}_{0,p}$ if and only if $\phi(z)=az+b$ for some $a,b\in\mathbb{C}$.
\end{theorem}
\begin{proof} Note that (cf. \cite{XiaoX, XiaoY}) if $p\in [0,2)$ then

$$
\begin{cases}
f_b(z)=(1-|b|^2)^{\frac{1+p}{2}}/(1-\bar{b}z)\Longrightarrow \sup_{b\in\mathbb D}\|f_b\|_{\mathcal{CA}_p,\ast}<\infty;\\
f\in\mathcal{CA}_{P}\Rightarrow|f'(z)|\lesssim\frac{||f||_{\mathcal{CA}_{P}}}{(1-|z|^{2})^{\frac{3-p}{2}}}\ \ \forall\ \ z\in\mathbb D.
\end{cases}
$$

Firstly, if $\mathsf{S}^\phi$ is bounded on $\mathcal{CA}_{p}$, then an application of Lemma \ref{l24} yields that $\mathsf{S}^\phi$ is bounded on $\mathcal{CA}_{0,p}$.

Secondly, if $\mathsf{S}^\phi$ is bounded on $\mathcal{CA}_{p}$, then for $f\in\mathcal{CA}_{p}$ one has 
$$
|(\mathsf{S}^\phi(f))'(z)|=
|\phi'(f(z))||f'(z)|\lesssim\frac{\|f\|_{\mathcal{CA}_{p}}}{(1-|z|^{2})^{\frac{3-p}{2}}}\quad\forall\quad z\in\mathbb{D}.
$$ 
Choosing the following $\mathcal{CA}_{p}$-function
$$
f_{\theta,b}(z)=e^{i\theta}(1-|b|^2)^{\frac{1+p}{2}}/(1-\bar{b}z)
$$ 
in the last inequality, one gets  
$$|\phi'(f_{\theta,b}(z))||f'_{\theta,b}(z)|=|\phi'(f_{\theta,b}(z))|\left(
\frac{|b|(1-|b|^{2})^{\frac{1+p}{2}}}{|1-\bar{b}z|^{2}}\right)\lesssim\frac{\|f_{\theta,b}\|_{\mathcal{CA}_{p}}}{(1-|z|^{2})^{\frac{3-p}{2}}}.
$$
Note that 
$$
\sup\limits_{\theta,b}\|f_{\theta,b}\|_{\mathcal{CA}_{p}}<\infty.
$$ 
So there is a positive $M$ independent of $(\theta,b)$ such that
$$
\sup\limits_{\theta,b}\|f_{\theta,b}\|_{\mathcal{CA}_{p}}\leq M.
$$
In particular, setting $b=z$ yields
$$|\phi'(f_{\theta,b}(b))||f'_{\theta,b}(b)|=\sup\limits_{b\in\mathbb{D}}|\phi'(f_b(b))||b|\leq 	M\implies |\phi'(e^{i\theta}(1-|b|^{2})^{\frac{p-1}{2}})|\leq M\quad\forall\quad b\in\mathbb D.
$$
Now letting $|b|\rightarrow 1$ in the last estimate and noticing $p\in [0,1)$, we obtain that 
the entire function $\phi$ is bounded on $\mathbb C$. An application of the maximum principle yields that $\phi$ must be a linear function.

Thirdly, if $\phi(z)=
az+b$ for some $a,b\in\mathbb C$, then $\phi'(z)=a$, and hence $\mathsf{S}^{\phi}(f)=af+b$. Hence $\mathsf{S}^{\phi}$ is bounded on $\mathcal{CA}_{p}$.
\end{proof}

\subsection{Backward shift}\label{s32}
For a function $z\mapsto f(z)=\sum\limits_{n=0}^{\infty}
a_{n}z^n$ in $\mathcal{CA}_{p}$, the backward shift operator $\mathsf{S_b}f$ is defined as 
$$ 
\mathsf{S_b}(f)(z)=\frac{f(z)-f(0)}{z}=
a_1+a_{2}z+a_{3}z^{2}+....
$$
As well-known, the backward shift operator plays an important role in the general study of bounded linear operators on Hilbert spaces; see \cite {CimaR}. Moreover, the Hardy space $\mathcal{H}^{p}$ is invariant under $\mathsf{S_b}$; see \cite{DouglasSS} for $p>1$ and \cite{Aleksandrov} for $p\in(0,1]$. The backward shift operator for $\mathcal{BMOA}$ and Lipschitz spaces have been considered in \cite{Zhang}. Hence, it remains to handle the action of $\mathsf{S_b}$ on the analytic Morrey space $\mathcal{CA}_{0<p<1}$ and its little one. The following result indicates the behavior of $\mathsf {S}_b$ on $\mathcal{CA}_{p}$ is particularly good.

\begin{theorem}
\label{t32}
Let $p\in (0,1)$. Then $\mathsf{S_b}$ is a bounded operator on both 
$\mathcal{CA}_{p}$ and $\mathcal{CA}_{0,p}$.
\end{theorem}

\begin{proof} Since the polynomials are dense in $\mathcal{CA}_{0,p}$ (cf. Lemma \ref{l24}), it suffices to prove that boundedness of $\mathsf{S_b}$ on $\mathcal{CA}_{p}$. To do so, suppose $f\in\mathcal{CA}_{p}$.\\
The decay of $f'$ ensures 
$$
\int_{\mathbb{D}}|f'(z)|(1-|z|^{2})dm(z)\lesssim
\|f\|_{\mathcal{CA}_p}\int_{\mathbb{D}}(1-|z|^{2})^{\frac{p-1}{2}}dm(z)<\infty.
$$ 
Hence $f'$ is in the weighted Bergman space $L_{a}^{1}(\mathbb{D},(1-|z|^{2})dm(z))$ on $\mathbb D$. 
Using the reproducing kernel formula on \cite
[p.120]{Rudin}: 
$$
f'(z)=2\pi^{-1}\int_{\mathbb{D}}\frac{f'(w)(1-|w|^{2})}{(1-\bar{w}z)^{3}}dm(w)\quad\forall\quad z\in\mathbb D,
$$
one obtains
\begin{align*}
&\mathsf{S_b}(f)(z)\\
&=\int_{0}^{1}f'(tz)dt\\
&=2\pi^{-1}\int_{\mathbb{D}}f'(w)(1-|w|^{2})\left(\int_{0}^{1}\frac{dt}{(1-t\bar{w}z)^{3}}\right)dm(w). 
\end{align*}
Differentiating the above equality with respect to $z$, one gets
\begin{equation}\label{eq31}
(\mathsf{S_b}(f))'(z)=6\pi^{-1}\int_{\mathbb{D}}\left(\int_{0}^{1}
\frac{t\bar{w}}{(1-t\bar{w}z)^{4}}dt
\right)f'(w)(1-|w|^{2})dm(w).
\end{equation}
Also, it is not hard to estimate
\begin{equation}\label{eq32}
\left| \int_{0}^{1}
\frac{t\bar{w}}{(1-t\bar{w}z)^{4}}dt\right|\lesssim
\frac{1}{|1-\bar{w}z|^{3}}.
\end{equation}
A combination of (\ref{eq31}) and (\ref{eq32}) yields
$$
|(\mathsf{S_b}(f))'(z)|\lesssim\int_{\mathbb{D}}\frac{|f'(w)|(1-|w|^{2})}{|1-\bar{w}z|^{3}}dm(w).
$$
Note that $f\in{\mathcal{H}^{2,p}}$. So $|f'(z)|^{2}(1-|z|^2)dm(z)$ is a $p$-Carleson measure. From the  Lemma \ref{l22} it follows that
$|(\mathsf{S_b}(f))'(z)|^2(1-|z|^2)dm(z)$ is also a $p$-Carleson measure. It is obvious that $\mathsf{S_b}(f)$ is analytic on $\mathbb{D}$. Thus, 
$$
\mathsf{S_b}(f)\in\mathcal{CA}_{p}\quad\hbox{with}\quad\|\mathsf{S_b}(f)\|_{\mathcal{CA}_{p}}\lesssim\|f\|_{\mathcal{CA}_{p}}.
$$

\end{proof}

\subsection{Schwarzian derivative}\label{s33}
 Let $f$ be a conformal mapping from $\mathbb{D}$ into a simply connected domain in $\mathbb C$. We say that $f(\mathbb{D})$ is a $\mathcal X$ domain whenever $\log f'$ belongs to an analytic function space $\mathcal X$ on $\mathbb D$. 

Recall that $\mathcal{B}$ is the Bloch space of all analytic functions $f$ in $\mathbb{D}$ with 
$$
\sup_{z\in\mathbb D}(1-|z|^{2})|f'(z)|<\infty
$$ 
and $\mathcal{B}_{0}$ is the little Bloch space consisting of all functions $f\in\mathcal{B}$ with 
$$
\lim\limits_{|z|\rightarrow 1}(1-|z|^{2})|f'(z)|=0.
$$
Pommerenke \cite {Pommerenke} proved that $f\in\mathcal{B}$ if and only if there is a 
constant $a\in\mathbb{C}$ and a univalent function 
$f$ such that $g=a\log f'$. So, it is interesting to characterize such a univalent $f$ that $\log f'\in\mathcal{X}$ and then to establish whether this condition implies some nice geometric properties of the image domain $f(\mathbb D)$. To be more precise, recall that the Schwarzian derivative of a univalent (analytic) function $f$ on $\mathbb D$ is defined by 
$$
\mathsf{S}(f)=\left(\frac{f''}{f'}\right)'-\frac{1}{2}\left(\frac{f''}{f'}\right)^{2}.
$$
This fully nonlinear operator plays an important role in geometric function theory, conformal field theory, differential equations and others. Interestingly, if $f$ is univalent on $\mathbb{D}$, then 
$$
|f''(z)/f'(z)|(1-|z|^{2})\leq 6\quad\&\quad
(1-|z|^{2})^{2}|\mathsf{S}(f)(z)|\leq 6\quad\forall\quad z\in\mathbb D.
$$
Conversely, if 
$$
|f''(z)/f'(z)|(1-|z|^{2})\leq 1\quad\hbox{or}\quad
(1-|z|^{2})^{2}|\mathsf{S}(f)(z)|\leq 2\quad\forall\quad z\in\mathbb D,
$$
then $f$ is univalent on $\mathbb{D}$, for more details see \cite {Pommerenke2}. Moreover, $\mathsf{S}(f)$ vanishes identically if and only if $f$ is a M\"obius mapping. Recently, there have been some results linking the Schwarzian derivative of a univalent analytic function to the characterizion of some analytic function spaces. According to Astala-Zinsmeister \cite {AstalaZ} and P\'erez-Gonz\'alez-R\"atty\"a \cite{PerezR}, one has
$$
\begin{cases}
\log f'\in\mathcal{BMOA}\Leftrightarrow |\mathsf{S}(f)(z)|^{2}(1-|z|^{2})^{3}dm(z)\ \ \hbox{is\ a\ bounded\ Carleson\ measure\ on}\ \mathbb{D};\\
\log f'\in\mathcal{VMOA}\Leftrightarrow |\mathsf{S}(f)(z)|^{2}(1-|z|^{2})^{3}dm(z)\ \ \hbox{is\ a\ vanishing\ Carleson\ measure\ on}\ \mathbb{D}.
\end{cases}
$$  
Further, $Q_{p}$-domains and even more general $F(p,q,s)$-domains were studied in \cite{Jin, PauP, Xiao2, Zorboska}. The above review actually leads to a consideration of the case of analytic Campanato spaces.

\begin{theorem}
\label{t33}
Let $f$ be a univalent function in $\mathbb{D}$. 

\begin{itemize}

\item {Case $p\in (0,1)$}:\ if $\log f'\in {\mathcal{CA}_{p}}$ ($\log f'\in {\mathcal{CA}_{0,p}}$) then $|\mathsf{S}(f)(z)|^{2}(1-|z|^{2})^{3}dm(z)$ is a bounded (vanishing) $p$-Carleson measure and $\log f'\in {\mathcal{B}}$. Conversely, if $|\mathsf{S}(f)(z)|^{2}(1-|z|^{2})^{3}dm(z)$ is a bounded (vanishing) $p$-Carleson measure and $\log f'\in {\mathcal{B}}_{0}$, then  $\log f'\in {\mathcal{CA}_{p}}$ ($\log f'\in {\mathcal{CA}_{0,p}}$).

\item {Case $p\in (1,3)$}: $\log f'\in {\mathcal{CA}_{p}}$ ($\log f'\in {\mathcal{CA}_{0,p}}$) if and only if $|\mathsf{S}(f)(z)|^{2}(1-|z|^{2})^{3}dm(z)$ is a bounded (vanishing) $p$-Carleson measure and $\log f'\in {\mathcal{B}}_{0}$.
\end{itemize}
\end{theorem}

\begin{proof} It is enough to verify boundedness.

 Under $p\in (0,1)$, we make the following consideration. For simplicity, let 
 $$
 \mathsf{P}(f)(z)=(\log f'(z))'=\frac{f''(z)}{f'(z)}.
 $$ 
 Then  
$$
\mathsf{S}{(f)}(z)=(\mathsf{P}(f))'{(z)}-\frac{1}{2}(\mathsf{P}(f)(z))^{2}.
$$ 
Note that $\log f'\in\mathcal{CA}_{p}$ and 
 $|\mathsf{P}(f)(z)|^{2}(1-|z|^{2})dm(z)$ is a $p$-Carleson measure if and only if  $|(\mathsf{P}(f))'(z)|^{2}(1-|z|^{2})^{3}dm(z)$ is a $p$-Carleson measure. So, applying the Cauchy-Schwarz inequality, one obtains
 $$
 |\mathsf{S}(f)(z)|^{2}\leq 2|(\mathsf{P}(f))'{(z)}|^{2}+|\mathsf{P}(f)(z)|^{4}\lesssim |(\mathsf{P}(f))'{(z)}|^{2}+|\mathsf{P}(f)(z)|^{4}.
 $$
 Using Pommerenke's result in \cite {Pommerenke}, one always has $\log f'\in\mathcal{B}$, that is 
 $$
 \sup_{z\in\mathbb D}|\mathsf{P}(f)(z)|(1-|z|^{2})=\sup_{z\in\mathbb D}|(\log f'(z))'|(1-|z|^{2})<\infty.
 $$  
 This in turns implies that
 $$
 |\mathsf{S}(f)(z)|^{2}(1-|z|^{2})^{3}\lesssim |(\mathsf{P}(f))'{(z)}|^{2}(1-|z|^{2})^{3}+|\mathsf{P}(f){(z)}|^{2}(1-|z|^{2}),
 $$
thereby giving that $|\mathsf{S}(f)(z)|^{2}(1-|z|^{2})^{3}dm(z)$ is a bounded $p$-Carleson measure.

For the converse part, set 
$$
I_{a}=\int_{\mathbb{D}}|(\mathsf{P}(f))'(z)|^{2}(1-|z|^{2})^{3}|\sigma_{a}'(z)|^{p}dm(z).
$$
Note that 
$$
\mathsf{S}(f)(z)=(\mathsf{P}(f))'{(z)}-\frac{1}{2}(\mathsf{P}(f)(z))^{2}.
$$
So 
\begin{equation}\label{e35}
I_{a}\leq  \int_{\mathbb{D}}\left(2|\mathsf{S}(f)(z)|^{2}(1-|z|^{2})^{3}|+\frac{1}{2}|\mathsf{P}(f)(z)|^{4}(1-|z|^{2})^{3}\right)|\sigma_{a}'(z)|^{p}dm(z).
\end{equation}
If $\log f'\in\mathcal{B}_{0}$, then for any $\epsilon>0$, there exsits $0<r_{\epsilon}<1$ such that $$
|\mathsf{P}(f)(z)|(1-|z|^{2})<\epsilon\quad\hbox{as}\quad |z|>r_{\epsilon}.
$$
Hence, there exsists some $\kappa>0$ depending only on $p$ such that 
\begin{align*} 
&\int_{|z|>r_{\epsilon}}|\mathsf{P}(f)(z)|^{4}(1-|z|^{2})^{3}|\sigma_{a}'(z)|^{p}dm(z)\\
&\leq {\epsilon}^{2}\int_{\mathbb{D}}|\mathsf{P}{(f)}(z)|^{2}(1-|z|^{2})|\sigma_{a}'(z)|^{p}dm(z)\\
&\leq {\epsilon}^{2}\kappa\int_{\mathbb{D}}|(\mathsf{P})(f))'(z)|^{2}(1-|z|^{2})^{3}|\sigma_{a}'(z)|^{p}dm(z)\\
&={\epsilon}^{2}\kappa I_{a}.
\end{align*}
At the same time, note that 
$$
|\mathsf{P}(f)(z)|(1-|z|^{2})\leq 6\quad\&\quad (1-|z|^{2})|\sigma_{a}'(z)|=1-|\sigma_{a}(z)|^{2}.
$$
So one obtains 
\begin{align*} 
&\int_{|z|\leq r_{\epsilon}}|\mathsf{P}(f)(z)|^{4}(1-|z|^{2})^{3}|\sigma_{a}'(z)|^{p}dm(z)\\
&\leq {6}^{4}\int_{|z|\leq r_{\epsilon}}(1-|z|^{2})^{-1}|\sigma_{a}'(z)|^{p}dm(z)\\
&\leq {6}^{4}\int_{|z|\leq r_{\epsilon}}(1-|z|^{2})^{-1-p}(1-|\sigma_{a}(z)|^{2})^{p}dm(z)\\
&\leq \frac{ {6}^{4}}{(1-r_{\epsilon}^{2})^{1+p}}.
\end{align*}
Hence 
\begin{equation}\label{e36}
\int_{\mathbb{D}}|\mathsf{P}(f)(z)|^{4}(1-|z|^{2})^{3}|\sigma_{a}'(z)|^{p}dm(z)\leq {\epsilon}^{2}\kappa I_{a}+\frac{ {6}^{4}}{(1-r_{\epsilon}^{2})^{1+p}}.
\end{equation}
Upon choosing $0<\epsilon$ to be so small that $\epsilon<\sqrt{2/\kappa}$, one gets from (\ref{e35}) and (\ref{e36}) that 
 $$
 (2-{{\epsilon}^{2}\kappa})I_a\leq 4\int_{\mathbb{D}}|\mathsf{S}(f)(z)|^{2}(1-|z|^{2})^{3}|\sigma_{a}'(z)|^{p}dm(z)+\frac{ {6}^{4}}{(1-r_{\epsilon}^{2})^{1+p}}.$$
Since $|\mathsf{S}(f)(z)|^{2}(1-|z|^{2})^{3}dm(z)$ is a bounded $p$-Carleson measure and $2-{\epsilon}^{2}\kappa>0$, it follows that
$$
\sup\limits_{a\in \mathbb{D}}\int_{\mathbb{D}}|(\mathsf{P}(f))'(z)|^{2}(1-|z|^{2})^{3}|\sigma_{a}'(z)|^{p}dm(z)<\infty.
$$
Hence $|(\mathsf{P}(f))'(z)|^{2}(1-|z|^{2})^{3}dm(z)$ is a bounded $p$-Carleson measure. Consequently, $\log f'\in{\mathcal{CA}_{p}}$.

Under $p\in (1,3)$, one has that ${\mathcal{CA}_{p}}$ is equal to $\mathcal A_{\frac{p-1}{2}}$ comprising all analytic functions on $\mathbb D$ with
 $$
\sup\limits_{z\in\mathbb{D}}(1-|z|^{2})^{\frac{3-p}{2}}|f'(z)|<\infty.
$$ 
Hence ${\mathcal{CA}_{p}}\subset\mathcal{B}_{0}$. Therefore, the argument for the case $p\in (0,1)$ can be utilized to complete the proof of the case $p\in (1,3)$.
\end{proof}


\begin{thebibliography}{99}

\bibitem{Adams} D. R. Adams, {A note on Choquet integrals with respect to Hausdorff capacity.} Function spaces and applications. Lund, 1986, 115-124, Lecture Notes in Math., 1302, Springer, Berlin, 1988. 

\bibitem{AX} D. R. Adams and J. Xiao, Morrey spaces in harmonic analysis. {\it Ark. Mat.}50(2012)201-230.


\bibitem{Aleksandrov} A. B. Aleksandrov, Invariant subspaces of the backward shift operator in the space $\mathcal{H}^{p}\,(p\in(0,1))$ (Russian), Investigations on linear operators and the theory of functions, IX. {\it Zap. Nauchn. Leningrad Otdel. Mat. Inst. Steklov (LOMI).} 92(1979) 7-29.

\bibitem{AlvarezMV} V. Alvarez, A. Marquez and D. Vukotic, Superposition operators between the Bloch space and Bergman spaces. {\it  Ark. Mat.} 42(2004) 205-216.

\bibitem{AppellZ} J. Appell and P. P. Zabrejko, Nonlinear superposition operators. {\it  Nonlinear superposition operators.} Cambridge University Press, Cambridge, 1990.

\bibitem{AstalaZ}
K. Astala and M. Zinsmeister, Teichm\"uller spaces and BMOA. {\it Math. Ann.} 289(1991) 613-625.

\bibitem{AxlerS} S. Axler and J. Shapiro, Putnam's theorem, Alexander's spectral area estimate, and VMO. {\it Math. Ann.} 271(1985) 161-183.

\bibitem{CarmonaC} J. Carmona and J. Cuf\'i, On the distance of an analytic function to VMO. {\it J. London Math. Soc.} 34(1986) 52-66.  

\bibitem{CascanteFO} C. Cascante, J. F\'abrega and J. M. Ortega, The corona theorem in weighted Hardy and Morrey spaces. {\it Ann. Scuola. Norm. Sup. Pisa.} Doi 10.2422/2036-2145.201202-006.

\bibitem{CimaR} J. A. Cima and W. T. Ross, {\it The backward shift on the Hardy space.} American Mathematical Society, 2000.


\bibitem{Devito} C. L. Devito,  {\it Functional Analysis.} Pure and Applied Mathematics, 38 Academic Press, New York, 1970.

\bibitem{DouglasSS} R. G. Douglas, H. S. Shapiro and A. L. Shields, On cyclic vectors of the backward shift. {\it Ann. Inst. Fourier (Grenoble)} 20(1970) 37-76.

\bibitem{DurenRS} P. L. Duren, B. W. Romberg and A. L. Shields, Linear functionals on $\mathcal{H}^{p}$ spaces with $0<p<1$. {\it J. Reine Angew. Math.} 238(1969) 32-60.

\bibitem{FeffermanS} C. Fefferman and E. M. Stein. $\mathcal{H}^{p}$ spaces of several variables. {\it Acta. Math.} 129(1972) 137-193.

\bibitem{GirelaAM} D. Girela and M. A. Marquez, Superposition operators between $Q_{p}$ spaces and Hardy spaces. {\it J. Math. Anal. Appl.} 364(2010) 463-472.

\bibitem{Jin} J. Jin, A note on $Q_p$ domains. {\it J. Math. Anal. Appl.} 406(2013) 511-518.

\bibitem{Kalton} N. J. Kalton, Spaces of Lipschitzand Holder functions and their applications. {\it Collect. Math.} 5(2004)  465-474.

\bibitem{LiLL}P. Li, J. Liu and Z. Lou, Integral operators on analytic Morrey spaces.  {\it Sci. China Math.} 57(2014) DOI 10.1007/s11425-000-0000-0.


\bibitem{PauP} J. Pau and J. Pel\'aez, Logarithms of the derivative of univalent functions in $Q_p$ spaces. {\it J. Math. Anal. Appl.} 350(2009) 184-194.



\bibitem{Perfekt} K. Perfekt, Duality and distance formulas in spaces defined by means of oscillation. {\it Ark. Mat.} 51(2013) 345-361.

\bibitem{PerezR} F. P\'erez-Gonz\'alez and G. R\"atty\"a, Dirichlet and VMOA domains via Schwarzian derivative. {\it J. Math. Anal. Appl.} 359(2)(2009) 184-194.

\bibitem{Pommerenke} C. Pommerenke, On Bloch functions. {\it J. London Math. Soc.} 2(1970) 689-695.

\bibitem{Pommerenke2} C. Pommerenke, {\it boundary behaviour of conformal maps.} Springer-Verlag, Berlin, 1992.


\bibitem{Rudin} W. Rudin, {\it Function Theory in the Unit Ball of $\mathbb{C}^{n}.$} Springer-Verlag, New York, 1980.


\bibitem{StegengaS} D. Stegenga and K. Stephenson, Sharp geometric estimates of the distance to VMOA. The Madison Symposium on Complex Analysis (Madison, WI, 1991), {\it Contemp. Math.} 137(2004) 421-432.

\bibitem{WX} K. J. Wirths and J. Xiao, Recognizing $Q_{p,0}$ functions per Dirichlet space structure. {\it Bull. Belg. Math. Soc.} 8(2001) 47-59.

\bibitem{WuX} Z. Wu and C. Xie, $Q$ spaces and Morrey spaces. {\it J. Funct. Anal.} 201(2003) 282-297.

\bibitem{Xiao1} J. Xiao, {\it Holomorphic Q Classes}. Lecture Notes in Math. 1767, Springer-Verlag, Berlin, 2001.

\bibitem{Xiao2} J. Xiao, {\it Geometric $Q_p$ Functions}. Birkhauser-Verlag, 2006.

\bibitem{XiaoX} J. Xiao and W. Xu, Composition operators between analytic Campanato space. {\it J. Geom. Anal.} Doi 10.1007/s12220-012-9349-6.

\bibitem{XiaoY} J. Xiao and C. Yuan, Analytic Campanato spaces and their compositions. {\it arXiv:1303.5032v2[math.CV]} 7May2013.

\bibitem{Xu} W. Xu, Superposition operators on Bloch-type spaces. {\it  Comput. Methods Funct. Theory.} 7(2007) 501-507.

\bibitem{Zhang} X. Zhang,  Solvability of Gleason's problem on the space $F(p,q,s)$  with several complex variables.  {\it Chinese Ann. Math. Ser. A}.  31 (2010)221-228.

\bibitem{Zhu} K. Zhu, {\it Operator Theory in Function Spaces}. Math. Surveys and Monographs, Vol. 138, Amer. Math. Soc., 2007.


\bibitem{Zorboska} N. Zorboska, Schwarzian derivative and general Besov-type domains. {\it J. Math. Anal. Appl.} 379(2011) 48-57.

\end{thebibliography}
\end{document}